\documentclass{amsart}

\usepackage{verbatim}
\usepackage{amssymb}
\usepackage{upref}
\usepackage[all]{xy}
\usepackage{color}

\allowdisplaybreaks

\newtheorem{thm}{Theorem}[section]

\newtheorem{cor}[thm]{Corollary}
\newtheorem{prop}[thm]{Proposition}

\theoremstyle{definition}

\theoremstyle{remark}
\newtheorem{rem}[thm]{Remark}
\newtheorem*{rem*}{Remark}
\newtheorem{step}{Step}

\numberwithin{equation}{section}

\newcommand{\secref}[1]{Section~\textup{\ref{#1}}}
\newcommand{\thmref}[1]{Theorem~\textup{\ref{#1}}}
\newcommand{\corref}[1]{Corollary~\textup{\ref{#1}}}
\newcommand{\propref}[1]{Proposition~\textup{\ref{#1}}}

\newcommand{\C}{\mathbb C}

\newcommand{\CC}{\mathcal C}

\newcommand{\KK}{\mathcal K}
\newcommand{\HH}{\mathcal H}
\newcommand{\GG}{\mathcal G}
\newcommand{\LL}{\mathcal L}

\newcommand{\EE}{\mathcal E}

\newcommand{\BB}{\mathcal B}

\newcommand{\XX}{\mathcal X}
\newcommand{\YY}{\mathcal Y}

\renewcommand{\AA}{\mathcal A}

\newcommand{\midtext}[1]{\quad\text{#1}\quad}
\newcommand{\righttext}[1]{\quad\text{#1 }}

\renewcommand{\)}{\textup)}

\DeclareMathOperator{\aut}{Aut}
\DeclareMathOperator{\ad}{Ad}

\DeclareMathOperator{\supp}{supp}

\DeclareMathOperator*{\spn}{span}
\DeclareMathOperator*{\clspn}{\overline{\spn}}

\newcommand{\<}{\langle}
\renewcommand{\>}{\rangle}

\newcommand{\under}{\backslash}
\newcommand{\inv}{^{-1}}
\renewcommand{\bar}{\overline}
\newcommand{\what}{\widehat}
\newcommand{\wilde}{\widetilde}

\newcommand{\lt}{\textup{lt}}

\begin{document}
\title[Fell bundles and imprimitivity theorems]{Fell bundles and imprimitivity theorems: towards a universal generalized fixed point algebra}

\author[Kaliszewski]{S. Kaliszewski}
\address{Department of Mathematics and Statistics
\\Arizona State University
\\Tempe, Arizona 85287}
\email{kaliszewski@asu.edu}

\author[Muhly]{Paul S. Muhly}
\address{Department of Mathematics
\\The University of Iowa
\\Iowa City, IA 52242}
\email{pmuhly@math.uiowa.edu}

\author[Quigg]{John Quigg}
\address{Department of Mathematics and Statistics
\\Arizona State University
\\Tempe, Arizona 85287}
\email{quigg@asu.edu}

\author[Williams]{Dana P. Williams}
\address{Department of Mathematics
\\Dartmouth College
\\Hanover, NH 03755}
\email{dana.williams@dartmouth.edu}



\date{June 23, 2012}

\begin{abstract}
  We apply the One-Sided Action Theorem from the first paper in this
  series to prove that Rieffel's Morita equivalence between the
  reduced crossed product by a proper saturated action and the
  generalized fixed-point algebra is a quotient of a Morita
  equivalence between the full crossed product and a ``universal''
  fixed-point algebra.  We give several applications, to Fell bundles
  over groups, reduced crossed products as fixed-point algebras, and
  $C^*$-bundles.
\end{abstract}
\maketitle

\section{Introduction}
\label{intro}

The One-Sided Action Theorem \cite[Corollary~2.3]{kmqw2} goes as
follows: let $\AA\to\XX$ be a Fell bundle over a locally compact
groupoid, and let $G$ be a locally compact group.  If $G$ acts freely
and properly on $\AA$, then the Banach bundle $\AA\to\XX$ gives a
Yamagami equivalence between the semidirect-product Fell bundle
$\AA\rtimes G\to \XX\rtimes G$ and the orbit Fell bundle
$G\under\AA\to G\under\XX$.  In the current paper we will connect this
quotient equivalence with Rieffel's imprimitivity theorem for
generalized-fixed-point-algebras \cite[Corollary~1.7]{proper}.  By
\cite[Theorem~6.4]{mw:fell}, $\Gamma_c(\AA)$ completes to give a
$C^*(\AA\rtimes G)-C^*(G\under\AA)$ imprimitivity bimodule $X$.  By
\cite[Theorem~7.1]{kmqw1} there is an associated action $\alpha:G\to
\aut C^*(\AA)$, and an isomorphism
\[
C^*(\AA\rtimes G)\cong C^*(\AA)\rtimes_\alpha G,
\]
so $X$ may be viewed as a $C^*(\AA)\rtimes_\alpha G-C^*(G\under\AA)$
imprimitivity bimodule.

We will show that the action $\alpha$ is proper and saturated in
Rieffel's sense, so that Rieffel's theorem gives a
$C^*(\AA)\rtimes_{\alpha,r} G-C^*(\AA)^\alpha$ imprimitivity bimodule
$X_R$.  Since $C^*(\AA)\rtimes_{\alpha,r} G$ is a quotient of
$C^*(\AA)\rtimes_\alpha G$, it seems natural to guess that the
imprimitivity bimodule $X_R$ is a quotient of $X$, and we will verify
this this in \thmref{Rieffel} below.

Thus, in some sense $C^*(G\under\AA)$ can be regarded as a
``universal'', or ``full'' version of a generalized fixed-point
algebra, whereas Rieffel's fixed-point algebra is in some sense a
``reduced'' version.  More precisely, Rieffel's generalized
fixed-point algebra is Morita equivalent to a \emph{reduced} crossed
product, while our ``universal'' fixed-point algebra is Morita
equvalent to the associated full crossed product.

We begin in \secref{prelims} with some preliminaries on transformation
Fell bundles.

\secref{main} contains our main result on the Rieffel Surjection, and
we further show that the quotient map of our ``universal fixed-point
algebra'' $C^*(G\under\AA)$ onto the ``reduced one'' $C^*(\AA)^\alpha$
can be identified with the regular representation of
$C^*(G\under\AA)$.

In \secref{applications} we give three applications: (1) for a Fell
bundle $\BB$ over a locally compact group $G$, we show that (as has
surely been suspected by the cognoscenti) the regular representation
of $C^*(\BB)$ is a normalization of the canonical coaction, (2) the
reduced crossed product by a $C^*$-action can be viewed as a
generalized fixed-point algebra, and (3) for a $C^*$-bundle over a
space, the full and reduced fixed-point algebras coincide.

\section{Preliminaries}\label{prelims}

We adopt the conventions of \cite{kmqw1,kmqw2}.  All our Banach
bundles will be upper semicontinuous and separable, all our spaces and
groupoids will be locally compact Hausdorff and second countable, and
our groupoids will all have left Haar systems.

In \cite[Proposition~1.7]{kmqw1}, building upon a result in
\cite{mw:fell}, we proved that if $\AA\to\XX$ is a Fell bundle over a
groupoid and $X_0$ is a dense subspace of a Hilbert $C^*$-module $X$,
and if we are given linear maps $\{L_0(f):f\in\Gamma_c(\AA)\}$ on
$X_0$ such that
\begin{enumerate}
\item $L_0(f)L_0(g)=L_0(fg)$ for all $f,g\in\Gamma_c(\AA)$,

\item $\<L_0(f)x,y\>=\<x,L_0(f^*)y\>$ for all
  $f\in\Gamma_c(\AA),x,y\in X_0$,

\item $f\mapsto \<L_0(f)x,y\>$ is inductive-limit continuous for all
  $x,y\in X_0$, and

\item $\spn\{L_0(f)x:f\in\Gamma_c(\AA),x\in X_0\}$ is dense in $X$,
\end{enumerate}
then $L_0$ extends uniquely to a nondegenerate homomorphism
$L:C^*(\AA)\to \LL(X)$.
In item (iii), recall that inductive-limit continuity means the following:
for any compact $K\subset \XX$,
and for any net $\{f_i\}$ in $\Gamma_c(\AA)$ with every $f_i$ supported in $K$,
if $f_i\to 0$ uniformly then
$\<L_0(f_i)x,y\>\to 0$.
In some applications of this result we will
have $X=C^*(\BB)$ and $X_0=\Gamma_c(\BB)$ for some Fell bundle $\BB$
(where $C^*(\BB)$ is regarded as a Hilbert module over itself in the
canonical way).  Another special case which can arise is where $\XX$
is a trivial groupoid and $\AA$ is a single $C^*$-algebra $A$, in
which case (iii) means that $a\mapsto \<L(a)x,y\>$ is norm continuous
for all $x,y\in X_0$.

Given a Fell bundle $\BB\to\YY$ over a groupoid, and an action of
$\YY$ on a space $\Omega$, in \cite[Section~A.1]{kmqw2} we defined a
transformation Fell bundle $\BB*\Omega\to \YY*\Omega$. We will need to
know a little more about this construction here:


\begin{prop}\label{Phi}
  Let $\BB\to\YY$ be a Fell bundle over a groupoid, and let $\YY$ act
  on a space $\Omega$.  Then there are nondegenerate homomorphisms
  $\Phi:C^*(\BB)\to M(C^*(\BB*\Omega))$ and $\mu:C_0(\Omega)\to
  M(C^*(\BB*\Omega))$ such that if $f\in \Gamma_c(\BB)$, $\phi\in
  C_0(\Omega)$, and $a\in \Gamma_c(\BB*\Omega)$, then $\Phi(f)a$ and
  $\mu(\phi)a$ are the sections in $\Gamma_c(\BB*\Omega)$ determined
  by
  \begin{align}
    \label{Phi def}
    \bigl(\Phi(f)a\bigr)_1(y,u)&=\int_\YY f(x)a_1(x\inv
    y,u)\,d\lambda^{r(y)}(x).
    \\
    \label{mu def}
    \bigl(\mu(\phi)a\big)_1(y,u)&=\phi(y\cdot u)a_1(y,u).
  \end{align}
  Moreover,
  \begin{equation}\label{Phi mu}
    \Phi(f)\mu(\phi)=f\boxtimes \phi,
  \end{equation}
  where $f\boxtimes \phi$ is the section in $\Gamma_c(\BB*\Omega)$
  determined by
  \[
  (f\boxtimes \phi)_1(y,u)=f(y)\phi(u),
  \]
  and we have
  \begin{equation}\label{density}
    \clspn\{\Phi(C^*(\BB))\mu(C_0(\Omega))\}=C^*(\BB*\Omega).
  \end{equation}
\end{prop}

\begin{proof}
  To establish the existence of the nondegenerate homomorphisms $\Phi$
  and $\mu$, we will apply the above-mentioned extension result
  \cite[Proposition~1.7]{kmqw1} with $X=C^*(\BB*\Omega)$ and
  $X_0=\Gamma_c(\BB*\Omega)$.

  We begin with $\Phi$.  For each $f\in\Gamma_c(\BB)$, \eqref{Phi def}
  defines a linear map $\Phi_0(f)$ on $\Gamma_c(\BB*\Omega)$.  If we
  can show that
  \begin{enumerate}
  \item $\Phi_0$ is multiplicative,

  \item $\<\Phi_0(f)a,b\>=\<a,\Phi_0(f^*)b\>$,

  \item $f\mapsto \<\Phi_0(f)a,b\>$ is inductive-limit continuous, and

  \item $\spn\{\Phi_0(f)a:f\in \Gamma_c(\BB),a\in
    \Gamma_c(\BB*\Omega)\}$ is inductive-limit dense in
    $\Gamma_c(\BB*\Omega)$,
  \end{enumerate}
  then it will follow that $\Phi_0$ extends uniquely to a
  nondegenerate homomorphism $\Phi:C^*(\BB)\to M(C^*(\BB*\Omega))$.

  For (i), if $f,g\in \Gamma_c(\BB)$ and $a\in \Gamma_c(\BB*\Omega)$
  we have
  \begin{align*}
    &\bigl(\Phi_0(f)\Phi_0(g)a\bigr)_1(y,u) \\&\quad=\int_\YY f(x)
    \bigl(\Phi_0(g)a\bigr)_1(x\inv y,u)\,d\lambda^{r(y)}(x)
    \\&\quad=\int_\YY f(x) \int_\YY g(z)a_1(z\inv x\inv y,u)
    \,d\lambda^{r(x\inv y)}(z)\,d\lambda^{r(y)}(x)
    \\&\quad=\int_\YY\int_\YY f(x)g(x\inv z)a_1(z\inv y,u)
    \,d\lambda^{r(y)}(z)\,d\lambda^{r(y)}(x) \\&\quad=\int_\YY
    \int_\YY f(x)g(x\inv z)\,d\lambda^{r(y)}(x) a_1(z\inv
    y,u)\,d\lambda^{r(y)}(z) \\&\quad=\int_\YY (f*g)(z) a_1(z\inv
    y,u)\,d\lambda^{r(y)}(z) \\&\quad=\bigl(\Phi_0(f*g)a\bigr)_1(y,u).
  \end{align*}

  For (ii), if $f\in \Gamma_c(\BB)$ and $a,b\in \Gamma_c(\BB*\Omega)$
  we have
  \begin{align*}
    &\bigl(\<\Phi_0(f)a,b\>_{\Gamma_c(\BB*\Omega)}\bigr)_1(y,u)
    \\&\quad=\bigl((\Phi_0(f)a)^**b\bigr)_1(y,u) \\&\quad=\int_\YY
    (\Phi_0(f)a)^*_1(x,x\inv y\cdot u)b_1(x\inv y,u)
    \,d\lambda^{r(y)}(x) \\&\quad=\int_\YY (\Phi_0(f)a)_1(x\inv,y\cdot
    u)^*b_1(x\inv y,u) \,d\lambda^{r(y)}(x) \\&\quad=\int_\YY
    \left(\int_\YY f(z)a_1(z\inv x\inv,y\cdot u)
      \,d\lambda^{r(x\inv)}(z)\right)^* b_1(x\inv y,u)
    \,d\lambda^{r(y)}(x) \\&\quad=\int_\YY \int_\YY a_1(z\inv
    x\inv,y\cdot u)^*f(z)^* \,d\lambda^{s(x)}(z) b_1(x\inv y,u)
    \,d\lambda^{r(y)}(x) \\&\quad=\int_\YY \int_\YY a_1(z\inv,y\cdot
    u)^*f(x\inv z)^* b_1(x\inv y,u) \,d\lambda^{r(x)}(z)
    \,d\lambda^{r(y)}(x) \\&\quad=\int_\YY \int_\YY a_1(z\inv,y\cdot
    u)^*f(x\inv z)^* b_1(x\inv y,u) \,d\lambda^{r(y)}(z)
    \,d\lambda^{r(y)}(x) \\&\quad=\int_\YY \int_\YY a_1(z\inv,y\cdot
    u)^*f(x\inv z)^* b_1(x\inv y,u) \,d\lambda^{r(y)}(x)
    \,d\lambda^{r(y)}(z) \\&\quad=\int_\YY \int_\YY a_1(z\inv,y\cdot
    u)^*f(x\inv z)^* b_1(x\inv y,u) \,d\lambda^{r(z)}(x)
    \,d\lambda^{r(y)}(z) \\&\quad=\int_\YY \int_\YY a_1(z\inv,y\cdot
    u)^*f(x\inv)^* b_1(x\inv z\inv y,u) \,d\lambda^{s(z)}(x)
    \,d\lambda^{r(y)}(z) \\&\quad=\int_\YY a_1(z\inv,y\cdot u)^*
    \int_\YY f^*(x) b_1(x\inv z\inv y,u) \,d\lambda^{r(z\inv y)}(x)
    \,d\lambda^{r(y)}(z) \\&\quad=\int_\YY a_1(z\inv,y\cdot u)^*
    \bigl(\Phi_0(f^*)b\bigr)_1(z\inv y,u) \,d\lambda^{r(y)}(z)
    \\&\quad=\bigl(a^**(\Phi_0(f^*)b)\bigr)_1(y,u)
    \\&\quad=\bigl(\<a,\Phi_0(f^*)b\>_{\Gamma_c(\BB*\Omega)})\bigr)_1(y,u).
  \end{align*}

  For (iii), let $K\subset \YY$ be compact, and let
  \[
  \Gamma_K(\BB)=\{f\in \Gamma_c(\BB):\supp f\subset K\}.
  \]
  For $f\in \Gamma_K(\BB)$ and $a,b\in \Gamma_c(\BB*\Omega)$, the
  inner product $\<\Phi_0(f)a,b\>$ in this part is to be interpreted
  in $C^*(\BB*\Omega)$.  It suffices to show that for fixed
  $a,b\in\Gamma_c(\BB*\Omega)$ the linear map
  \[
  f\mapsto \<\Phi_0(f)a,b\> :\Gamma_K(\BB)\to C^*(\BB*\Omega)
  \]
  is bounded when $\Gamma_K(\BB)$ is given the uniform norm
  $\|\cdot\|_u$, and for this it suffices to show that this linear map
  actually takes values in $\Gamma_P(\BB*\Omega)$ for some compact set
  $P\subset \YY*\Omega$ and is bounded when $\Gamma_P(\BB*\Omega)$ is
  given its uniform norm $\|\cdot\|_u$.

  Choose compact sets $L\subset \YY$ and $M\subset \YY^0*\Omega$ such
  that both $a$ and $b$ are supported in $L\times M$.  The computation
  in (ii) shows that for $a,b\in\Gamma_c(\BB*\Omega)$ we have
  \begin{align*}
    &\bigl(\<\Phi_0(f)a,b\>_{\Gamma_c(\BB*\Omega)}\bigr)_1(y,u)
    \\&\quad=\int_\YY \int_\YY a_1(z\inv,y\cdot u)^*f(x\inv z)^*
    b_1(x\inv y,u) \,d\lambda^{r(y)}(z) \,d\lambda^{r(y)}(x),
  \end{align*}
  so for the left-hand side to be nonzero we must have $u\in M$.  Then
  for the integration we can assume that $z\inv\in L$ and $x\inv z\in
  K$, so $x\in L\inv K\inv$, and that $x\inv y\in L$, so $y\in L\inv
  K\inv L$.  Thus $\<\Phi_0(f)a,b\>_{\Gamma_c(\BB*\Omega)}$ is
  supported in the compact set
  \[
  P:=(L\inv K\inv L)*M\subset \YY*\Omega.
  \]
  We have
  \begin{align*}
    &\Bigl\|\bigl(\<\Phi_0(f)a,b\>_{\Gamma_c(\BB*\Omega)}\bigr)_1(y,u)\Bigr\|
    \\&\quad\le \int_{L\inv K\inv}\int_{L\inv} \bigl\|a_1(z\inv,y\cdot
    u)\bigr\| \bigl\|f(x\inv z)\bigr\| \bigl\|b_1(x\inv y,u)\bigr\|
    \\&\quad\hspace*{1in} \,d\lambda^{r(y)}(z) \,d\lambda^{r(y)}(x)
    \\&\quad\le \|a\|_u\|f\|_u\|b\|_u \lambda^{r(y)}(L\inv
    K\inv)\lambda^{r(y)}(L\inv).
  \end{align*}
  Since the sets $L\inv K\inv$, $L\inv$, and $L\inv K\inv L$ are
  compact, by the properties of Haar systems there is a constant $c$
  such that
  \[
  \lambda^{r(y)}(L\inv K\inv)\lambda^{r(y)}(L\inv) \le c\righttext{for
    all}y\in L\inv K\inv L,
  \]
  giving the required boundedness.

  For (iv), first note that the properties of Banach bundles imply
  that
  \begin{equation}\label{box dense}
    \text{$\Gamma_c(\BB)\boxtimes C_0(\Omega)$ is inductive-limit dense in $\Gamma_c(\BB*\Omega)$.}
  \end{equation}
  Moreover, it is clear that if $f,g\in \Gamma_c(\BB)$ and $\phi\in
  C_c(\YY^0*\Omega)$ then
  \[
  \Phi_0(f)(g\otimes\phi)=(f*g)\boxtimes \phi.
  \]
  Since $\spn\{f*g:f,g\in \Gamma_c(\BB)\}$ is inductive-limit dense in
  $\Gamma_c(\BB)$, and since for fixed $\phi\in C_c(\YY^0*\Omega)$ the
  linear map
  \[
  g\mapsto g\boxtimes \phi:\Gamma_c(\BB)\to \Gamma_c(\BB*\Omega)
  \]
  is inductive-limit continuous, the required density follows.

  We have proved the existence of $\Phi$ satisfying \eqref{Phi def}.
  The proof of the existence of $\mu$ satisfying \eqref{mu def} is
  similar, but easier.  Again, for each $\phi\in C_0(\Omega)$
  \eqref{mu def} defines a linear map $\mu_0(\phi)$ on
  $\Gamma_c(\BB*\Omega)$, and we must verify appropriate versions of
  the properties (i)--(iv). Of these, (i) and (iii) are obvious, and
  (iv) follows from density of $\Gamma_c(\BB)\boxtimes
  C_c(\Omega)$. We give the computation for (ii):
  \begin{align*}
    \<\mu_0(f)a,b\>_1(y,u) &=\int_\YY (\mu_0(f)a)^*_1(x,x\inv y\cdot
    u)b_1(x\inv y,u)\,\lambda^{r(y)}(x) \\&=\int_\YY
    (\mu_0(f)a)_1(x\inv,y\cdot u)^*b_1(x\inv y,u)\,\lambda^{r(y)}(x)
    \\&=\int_\YY a_1(x\inv,y\cdot u)^*\bar{f(x\inv y\cdot u)}b_1(x\inv
    y,u)\,\lambda^{r(y)}(x) \\&=\int_\YY a^*_1(x,x\inv y\cdot
    u)(\mu_0(\bar f)b)_1(x\inv y,u)\,\lambda^{r(y)}(x)
    \\&=\<a,\mu_0(\bar f)b\>(y,u).
  \end{align*}

  Finally, \eqref{Phi mu} is a simple computation, and then
  \eqref{density} follows from \eqref{box dense}.
\end{proof}

\section{The Rieffel Surjection}
\label{main}

Our main result is the following application of the One-Sided Action
theorem:

\begin{thm}[Rieffel Surjection]\label{Rieffel}
  Let $p:\AA\to\XX$ be a Fell bundle over a locally compact groupoid,
  and let $G$ be a locally compact group.  Suppose that $G$ acts
  freely and properly on \(the left of\) $\AA$ by automorphisms, so
  that we also have an associated action $\alpha:G\to \aut C^*(\AA)$.
  Then there exist maps $\Upsilon$ and $\Phi$ such that
  \begin{equation*}
    (\Lambda,\Upsilon,\Phi): (C^*(\AA)\rtimes_\alpha
    G,X,C^*(G\under\AA))
    \to (C^*(\AA)\rtimes_{\alpha,r} G,X_R,C^*(\AA)^\alpha)
  \end{equation*}
  is a surjection of imprimitivity bimodules, where $\Lambda$ is the
  regular representation.
\end{thm}

The above theorem will be proven in the following equivalent form,
rephrased using the principal-bundle decomposition
\cite[Theorem~A.11]{kmqw2}:

\begin{thm}[Rieffel Surjection]\label{Rieffel transformation}
  Let $p:\BB\to\YY$ be a Fell bundle over a locally compact groupoid,
  and let $G$ be a locally compact group. Suppose that both $\YY$ and
  $G$ act on \(the left of\) a locally compact Hausdorff space
  $\Omega$, 
  that the action of $G$ is free and proper and commutes
  with the $\YY$-action, 
  and that the fibring map $\Omega\to \YY^0$ associated to the 
  $\YY$-action induces an identification of 
  $G\under \Omega$ with $\YY^0$,
  so that $G$ also acts freely and properly by
  automorphisms on the transformation Fell bundle $\BB*\Omega\to
  \YY*\Omega$, and we also have an associated action $\alpha:G\to\aut
  C^*(\BB*\Omega)$.  Then there exist maps $\Upsilon$ and $\Phi$ such
  that
  \begin{equation*}
    (\Lambda,\Upsilon,\Phi): (C^*(\BB*\Omega))\rtimes_\alpha
    G,X,C^*(\BB))
    \to (C^*(\BB*\Omega)\rtimes_{\alpha,r}
    G,X_R,C^*(\BB*\Omega)^\alpha)
  \end{equation*}
  is a surjection of imprimitivity bimodules, where $\Lambda$ is the
  regular representation.
\end{thm}

\begin{rem*}
  Obviously \thmref{Rieffel} implies \thmref{Rieffel transformation};
  to see that the two results are in fact equivalent, just use the
  principal-bundle decomposition
  \[
  \AA\cong \BB*\Omega
  \]
  from \cite[Theorem~A.11]{kmqw2}.
\end{rem*}

Our strategy for proving \thmref{Rieffel transformation} will be to
take $\Upsilon$ as a suitable extension of the identity map on
$\Gamma_c(\BB*\Omega)$.  This makes sense, since both imprimitivity
bimodules $X$ and $X_R$ are completions of $\Gamma_c(\BB*\Omega)$, in
the latter case because the action $\alpha$ of $G$ on
$C^*(\BB*\Omega)$ is saturated and proper with respect to the dense
*-subalgebra $\Gamma_c(\BB*\Omega)$.


We first verify that the homomorphism $\Phi$ from \propref{Phi} is the
one we want:

\begin{prop}\label{onto}
  With the hypotheses of \thmref{Rieffel transformation}, the
  homomorphism $\Phi:C^*(\BB)\to M(C^*(\BB*\Omega))$ of \propref{Phi}
  maps onto the generalized fixed-point algebra
  $C^*(\BB*\Omega)^\alpha$.
\end{prop}

\begin{proof}
  The proof will be rather long, and we break it into steps.

\begin{step}
  We first need to know that the generalized fixed-point algebra
  exists, and for the proof of \thmref{Rieffel transformation} we
  further want to know that this fixed-point algebra is Morita
  equivalent to the reduced crossed product; by
  \cite[Corollary~1.7]{proper}, we can accomplish this by showing that
  the action $\alpha$ of $G$ on $C^*(\BB*\Omega)$ is proper and
  saturated with respect to the dense *-subalgebra
  $\Gamma_c(\BB*\Omega)$.  Recall that there is a nondegenerate
  embedding of $C_0(\XX^0)$ in $M(C^*(\BB*\Omega))$ determined by
  \[
  (\phi\cdot a)(x)=\phi(r(x))a(x) \midtext{and} (a\cdot
  \phi)(x)=a(x)\phi(s(x))
  \]
  for $\phi\in C_0(\XX^0)$ and $a\in \Gamma_c(\BB*\Omega)$.  Moreover,
  this embedding is $G$-equivariant.  Therefore properness and
  saturatedness follow from \cite[Theorem~5.7]{integrable} and
  \cite[Lemma~4.1]{HRWsymmetric}.
\end{step}

\begin{step}
  For each $a\in \Gamma_c(\BB*\Omega)$, we will show that there exists
  a unique section $\Psi(a)\in \Gamma_c(\BB)$ such that
  \[
  \Psi(a)(y)=\int_G a_1(y,s\inv\cdot u)\,ds,
  \]
  where $u\in \Omega$ is any element satisfying $q(u)=s(y)$.

  It is clear that the value of the integral is a well-defined element
  of $B(y)$, because for any $u\in q\inv(s(y))$ the map $s\mapsto
  a_1(y,s\inv\cdot u)$ is in $C_c(G,B(y))$, and by left-invariance of
  the Haar measure on $G$ the value of the integral is independent of
  the choice of $u$. It is also clear that the integral is zero for
  $y$ outside the compact subset $q(\supp a)$ of $\YY$. It remains to
  see that $\Psi(a)$ is continuous.  For this purpose we first show
  continuity of the auxiliary function $\wilde\Psi(a):\XX\to
  \BB*\Omega$ defined by
  \[
  \wilde\Psi(a)(y,u)=\int_G \bigl(a_1(y,s\inv\cdot u),u\bigr)\,ds.
  \]
  Once we have shown this, it will be clear that in fact
  $\wilde\Psi(a)$ is a section of the bundle $p:\BB*\Omega\to \XX$,
  and that
  \[
  \wilde\Psi(a)_1(y,u)=\int_G a_1(y,s\inv\cdot u)\,ds.
  \]
  Since the function $\wilde\Psi(a)_1$ is obviously independent of the
  second variable, we will be able to conclude that
  \[
  \Psi(a)\circ q=\wilde\Psi(a),
  \]
  and hence that the function $\Psi(a)$ is continuous as well, because
  $q:\XX\to \YY$ is a quotient map.

  Fix $(y_0,u_0)\in \XX$, and choose a compact neighborhood $U$ of
  $(y_0,u_0)$ and then a function $g\in C_c(\XX)$ that is identially
  $1$ on $U$. Then $\wilde\Psi(a)=\wilde\Psi(a)g$ on $U$, so to show
  that $\tilde\Psi(a)$ is continuous at $(y_0,u_0)$ it suffices to
  show that $\wilde\Psi(a)g$ is continous.  Let $K=\supp g$, and
  define
  \[
  \Gamma_K(\BB*\Omega)=\{c\in \Gamma_c(\BB*\Omega):\supp c\subset K\},
  \]
  which is a Banach space with the sup norm.  Now define $\psi\in
  C_c(G,\Gamma_K(\BB*\Omega)$ by
  \[
  \psi(s)(y,u)=a(y,s\inv u)g(y,u).
  \]
  Then the integral $\int_G \psi(s)\,ds$ is norm-convergent in the
  Banach space $\Gamma_K(\BB*\Omega)$, and it is routine to check that
  it agrees with $\tilde\Psi(a)g$, and this completes Step~\thestep.
\end{step}

\begin{step}
  We recall that
  \[
  C^*(\BB*\Omega)^\alpha=\bar{E(\Gamma_c(\BB*\Omega))},
  \]
  where $E$ is the conditional expectation from \cite{cldx}, defined
  as follows: for $a\in \Gamma_c(\BB*\Omega)$, $E(a)$ is the unique
  element of $M(C^*(\BB*\Omega))$ such that for all $\omega\in
  C^*(\BB*\Omega)^*$ we have
  \[
  \omega(E(a))=\int_G \omega(\alpha_s(a))\,ds.
  \]
  We will show that if $a,b\in \Gamma_c(\BB*\Omega)$ then $E(a)b$
  coincides with the section in $\Gamma_c(\BB*\Omega)$ given by
  $\Phi\circ\Psi(a)b$.

  We first show that $E(a)b$ is an integral in $C^*(\BB*\Omega)$, more
  precisely
  \begin{equation}\label{integral}
    E(a)b=\int_G^{C^*(\BB*\Omega)} \alpha_s(a)b\,ds,
  \end{equation}
  where the superscript ``$C^*(\BB*\Omega)$'' on the integral sign
  indicates that this is a norm-convergent integral in
  $C^*(\BB*\Omega)$.  To verify \eqref{integral}, let $\omega\in
  C^*(\BB*\Omega)^*$, and define $b\cdot \omega\in C^*(\BB*\Omega)^*$
  by
  \[
  b\cdot \omega(a)=\omega(ab).
  \]
  Then
  \begin{align*}
    \omega\bigl(E(a)b\bigr) &=b\cdot\omega(E(a)) \\&=\int_G
    b\cdot\omega(\alpha_s(a))\,ds \\&=\int_G\omega(\alpha_s(a)b)\,ds
    \\&=\omega\left(\int_G^{C^*(\BB*\Omega)} \alpha_s(a)b\,ds\right),
  \end{align*}
  because $s\mapsto \alpha_s(a)b$ is in $C_c(G,C^*(\BB*\Omega))$.

  Our strategy is to identify the integral in \eqref{integral} with
  one in $\Gamma_c(\BB*\Omega)$.  Define $g:G\to \Gamma_c(\BB*\Omega)$
  by
  \[
  g(s)=\alpha_s(a)b.
  \]
  Then there exist compact sets $C\subset G$ and $K\subset\XX$ such
  that
  \[
  g(s)(x)=0\midtext{if}(s,x)\notin C\times K.
  \]
  We can view
  \[
  g\in C_c(G,\Gamma_K(\BB*\Omega)),
  \]
  and hence we can integrate this map, getting an element
  \[
  c:=\int_G^{\Gamma_K(\BB*\Omega)} g(s)\,ds
  \]
  of $\Gamma_K(\BB*\Omega)$.  Let $j:\Gamma_K(\BB*\Omega)\to A$ be the
  inclusion map.  Then $j$ is bounded, and we have
  \begin{align*}
    j(c) &=j\left(\int_G^{\Gamma_K(\BB*\Omega)} g(s)\,ds\right)
    \\&=\int_G^A j(g(s))\,ds \\&=E(a)b,
  \end{align*}
  showing that $E(a)b$ is the section in $\Gamma_c(\BB*\Omega)$ given
  by
  \[
  \bigl(E(a)b\bigr)(y,u)=\int_G \bigl(\alpha_s(a)b\bigr)(y,u)\,ds.
  \]

  To show that $E(a)b=\Phi\circ\Psi(a)b$, note that since evaluation
  at $(y,u)\in \XX$ is a continuous linear map of
  $\Gamma_K(\BB*\Omega)$ into the fibre $B(y)\times \{u\}$, it follows
  that
  \begin{align*}
    \bigl(E(a)b\bigr)_1(y,u) &=\int_G
    \bigl(\alpha_s(a)b\bigr)_1(y,u)\,ds \\&=\int_G\int_\YY
    \alpha_s(a)_1(x,x\inv y\cdot u)b_1(x\inv y,u)
    \,d\lambda^{r(y)}(x)\,ds \\&=\int_G\int_\YY a_1(x,s\inv\cdot x\inv
    y\cdot u)b_1(x\inv y,u) \,d\lambda^{r(y)}(x)\,ds
    \\&=\int_\YY\int_G a_1(x,s\inv\cdot x\inv y\cdot u)b_1(x\inv y,u)
    \,ds\,d\lambda^{r(y)}(x) \\&=\int_\YY \int_G a_1(x,s\inv\cdot
    x\inv y\cdot u)\,ds b_1(x\inv y,u)\,d\lambda^{r(y)}(x)
    \\&=\int_\YY \Psi(a)(x) b_1(x\inv y,u)\,d\lambda^{r(y)}(x)
    \\&=\bigl(\Phi(\Psi(a))b\bigr)_1(y,u).
  \end{align*}
  This completes Step~\thestep.
\end{step}

\begin{step}
  We show that $\Psi(\Gamma_c(\BB*\Omega))$ is inductive-limit dense
  in $\Gamma_c(\BB)$.

  Clearly $\Psi(\Gamma_c(\BB*\Omega))$ is a $C_0(\YY)$-module, so it
  suffices to show that its fibres are full, i.e., for $y\in\YY$ and
  $c\in B(y)$ we can find a section in $\Psi(\Gamma_c(\BB*\Omega))$
  whose value at $y$ is $c$.  Pick $b\in \Gamma_c(\BB)$ such that
  $b(y)=c$.  Next choose $u\in \Omega$ such that $q(u)=s(y)$, then
  choose a nonnegative function $g\in C_c(\Omega)$ such that
  \[
  \int_G g(s\inv\cdot u)\,ds=1.
  \]
  To see that such a $g$ exists, note that $s\mapsto s\inv\cdot u$ is
  a homeomorphism of $G$ onto the closed subset $G\cdot u$ of
  $\Omega$, and we can choose a nonnegative function $g_0\in
  C_c(G\cdot u)$ such that $\int_G g_0(s\inv\cdot x_0)\,ds=1$, and
  then use Tietze's theorem to extend $g_0$ to $g\in C_c(\Omega)$.
  Let $b\boxtimes g$ denote the element of $\Gamma_c(\BB*\Omega)$
  defined by
  \[
  (b\boxtimes g)(y,s)=(b(y)g(s),s).
  \]
  Then
  \begin{align*}
    \Psi(b\boxtimes g)(y) &=\int_G (b\boxtimes g)_1(y,s\inv\cdot
    u)\,ds \\&=\int_G b(y)g(s\inv\cdot u)\,ds \\&=b(y)\int_G
    g(s\inv\cdot u)\,ds \\&=b(y) \\&=c.
  \end{align*}
\end{step}

\begin{step}
  \propref{onto} now follows quickly from the above: we have a
  homomorphism $\Phi:C^*(\BB)\to M(C^*(\BB*\Omega))$ such that
  \[
  \Phi(\Psi(\Gamma_c(\BB*\Omega)))=E(\Gamma_c(\BB*\Omega)).
  \]
  Since $\Psi(\Gamma_c(\BB*\Omega))$ is dense in $C^*(\BB)$ via the
  composition of inclusions
  \[
  \Psi(\Gamma_c(\BB*\Omega))\hookrightarrow
  \Gamma_c(\BB)\hookrightarrow C^*(\BB),
  \]
  and since $E(\Gamma_c(\BB*\Omega))$ is dense in
  $C^*(\BB*\Omega)^\alpha$, we have
  $\Phi(C^*(\BB))=C^*(\BB*\Omega)^\alpha$.  \qedhere
\end{step}
\end{proof}


\begin{proof} [Proof of \thmref{Rieffel transformation}]
  We will show that the identity map on $\Gamma_c(\BB*\Omega)$ is
  compatible with the regular representation $\Lambda$ and the
  homomorphism $\Phi$ from \propref{Phi}.

  We begin by recalling the formulas associated with the imprimitivity
  bimodules $X$ and $X_R$.  For $X$, first recall the abstract
  formulas from \cite{mw:fell} for the imprimitivity bimodule
  associated to an equivalence bundle $\EE\to \Omega$ between Fell
  bundles $\BB\to \GG$ and $\CC\to \HH$:
  \begin{align*}
    (f\cdot \xi)(t)&=\int_\GG f(x)\cdot \xi(x\inv\cdot t)\,d\lambda_\GG^{\rho(t)}(x)\\
    {}_L\<\xi,\eta\>(x)&=\int_\HH {}_\BB\<\xi(x\cdot t\cdot h),\eta(t\cdot h)\>\,d\lambda_\HH^{\sigma(t)}(h)\\
    (\xi\cdot f)(t)&=\int_\HH \xi(t\cdot h)\cdot f(h\inv)\,d\lambda_\HH^{\sigma(t)}(h)\\
    \<\xi,\eta\>_R(h)&=\int_\GG \<\xi(x\inv\cdot t),\eta(x\inv\cdot
    t\cdot h\>_\CC\,d\lambda^{\rho(t)}(x),
  \end{align*}
  where in the second and fourth equations $t\in \Omega$ is any
  element satisfying $\rho(t)=s(x)$ and $\sigma(t)=r(h)$,
  respectively.

  In our context, we have an equivalence $\BB*\Omega\to \YY*\Omega$
  between Fell bundles $(\BB*\Omega)\rtimes G\to (\YY*\Omega)\rtimes
  G$ and $\BB\to \YY$. The left module action becomes
  \begin{align*}
    &(f\cdot \xi)(y,u) \\&\quad=\int_{(\YY*\Omega)\rtimes G}
    f(x,v,s)\cdot \xi\bigl((x,v,s)\inv)\cdot (y,u)\bigr)
    \,d\lambda_{(\YY*\Omega)\rtimes G}^{\rho(y,u)}(x,v,s)
    \\&\quad=\int_{(\YY*\Omega)\rtimes G} f(x,v,s)\cdot
    \xi\Bigl(\bigl(s\inv\cdot (x,v)\inv,s\inv\bigr)\cdot (y,u)\Bigr)
    \,d\lambda_{(\YY*\Omega)\rtimes G}^{(r(y),y\cdot u,e)}(x,v,s)
    \\&\quad=\int_{\YY*\Omega}\int_G f(x,v,s)\cdot
    \xi\Bigl(\bigl(s\inv\cdot (x\inv,x\cdot v),s\inv\bigr)\cdot
    (y,u)\Bigr) \,ds\,d\lambda_{\YY*\Omega}^{(r(y),y\cdot u)}(x,v)
    \\&\quad=\int_\YY\int_G f(x,x\inv y\cdot u,s)\cdot
    \xi\bigl((x\inv,s\inv\cdot x\cdot v,s\inv)\cdot (y,u)\bigr)
    \,ds\,d\lambda_\YY^{r(y)}(x) \\&\quad\hspace{1in}\text{because we
      must have $y\cdot u=x\cdot v$} \\&\quad=\int_\YY\int_G f(x,x\inv
    y\cdot u,s)\cdot \xi(x\inv y,s\inv\cdot u)
    \,ds\,d\lambda_\YY^{r(y)}(x).
  \end{align*}

  The left inner product becomes
  \begin{align*}
    &{}_L\<\xi,\eta\>(y,u,s) \\&\quad=\int_\YY {}_{(\BB*\Omega)\rtimes
      G}\Bigl\< \xi\bigl((y,u,s)\cdot (x,v)\cdot z\bigr),
    \eta\bigl((x,v)\cdot z\bigr) \Bigr\> \,d\lambda^{\sigma(x,v)}(z)
    \\&\quad\hspace{1in}\text{where $\rho(x,v)=s(y,u,s)$}
    \\&\quad=\int_\YY {}_{(\BB*\Omega)\rtimes G}\Bigl\<
    \xi\bigl((yx,s\cdot v)\cdot z\bigr), \eta(xz,z\inv\cdot v) \Bigr\>
    \,d\lambda^{s(x)}(z) \\&\quad=\int_\YY {}_{(\BB*\Omega)\rtimes
      G}\bigl\< \xi(yxz,z\inv\cdot s\cdot v), \eta(xz,z\inv\cdot v)
    \bigr\> \,d\lambda^{s(x)}(z) \\&\quad=\int_\YY
    {}_{(\BB*\Omega)\rtimes G}\bigl\< \xi(yxz,s\cdot z\inv\cdot v),
    \eta(xz,z\inv\cdot v) \bigr\> \,d\lambda^{s(x)}(z)
    \\&\quad=\int_\YY {}_{(\BB*\Omega)\rtimes G}\Bigl\<
    \bigl(\xi_1(yxz,s\cdot z\inv\cdot v),x\cdot z\inv\cdot v\bigr),
    \\&\quad\hspace{1in}\bigl(\eta_1(xz,z\inv\cdot v),z\inv\cdot
    v\bigr) \Bigr\> \,d\lambda^{s(x)}(z) \\&\quad=\int_\YY \bigl(
    \xi_1(yxz,s\cdot z\inv\cdot v)\eta_1(xz,z\inv\cdot v)^*,
    \\&\quad\hspace{1in}s\cdot p(\eta_1(xz,z\inv\cdot v))\cdot
    z\inv\cdot v,s \bigr) \,d\lambda^{s(x)}(z) \\&\quad=\int_\YY
    \bigl( \xi_1(yxz,s\cdot z\inv\cdot v)\eta_1(xz,z\inv\cdot v)^*,
    s\cdot xz\cdot z\inv\cdot v,s \bigr) \,d\lambda^{s(x)}(z)
    \\&\quad=\int_\YY \bigl( \xi_1(yxz,z\inv x\inv\cdot
    u)\eta_1(xz,s\inv\cdot z\inv x\inv\cdot u)^*, u,s \bigr)
    \,d\lambda^{s(x)}(z)
  \end{align*}
  because $\rho(x,v)=(r(x),x\cdot v,e)$ and $s(y,u,s)=(s(y),s\inv\cdot
  u,e)$, and thus
  \begin{align*}
    &{}_{(\BB*\Omega)\rtimes G}\<\xi,\eta\>_1(y,u,s) \\&\quad=\int_\YY
    \xi_1(yxz,z\inv x\inv\cdot u)\eta_1(xz,s\inv\cdot z\inv x\inv\cdot
    u)^*\,d\lambda^{s(x)}(z) \\&\quad=\int_\YY \xi_1(yz,z\inv\cdot
    u)\eta_1(z,s\inv\cdot z\inv\cdot u)^*\,d\lambda^{r(x)}(z)
    \\&\quad=\int_\YY \xi_1(yz,z\inv\cdot u)\eta_1(z,s\inv\cdot
    z\inv\cdot u)^*\,d\lambda^{s(y)}(z).
  \end{align*}

  The right module action becomes
  \begin{align*}
    (\xi\cdot f)(y,u) &=\int_\YY \xi\bigl((y,u)\cdot x\bigr)\cdot
    f(x\inv) \,d\lambda_\YY^{\sigma(y,u)}(x) \\&=\int_\YY
    \xi(yx,x\inv\cdot u)\cdot f(x\inv) \,d\lambda_\YY^{s(y)}(x),
  \end{align*}
  so
  \begin{align*}
    (\xi\cdot f)_1(y,u) &=\int_\YY \xi_1(yx,x\inv\cdot
    u)f(x\inv)\,d\lambda^{s(y)}(x).
  \end{align*}

  The right inner product becomes
  \begin{align*}
    \<\xi,\eta\>_R(y) &=\int_{(\YY*\Omega)\rtimes G} \Bigl\<
    \xi\bigl((x,u,s)\inv\cdot (z,v)\bigr), \\&\hspace{1in}
    \eta\bigl((x,u,s)\inv\cdot (z,v)\cdot y\bigr) \Bigr\>_\BB
    \,d\lambda^{\rho(z,v)}(x,u,s),
  \end{align*}
  where $(z,v)$ is any element of $\YY*\Omega$ such that
  $\sigma(z,v)=r(y)$.  Since $\sigma(z,v)=s(z)$, we can take
  $z=y\inv$.  We have
  \begin{align*}
    (x,u,s)\inv &=\bigl(s\inv\cdot (x,u)\inv,s\inv\bigr)
    \\&=\bigl(s\inv\cdot (x\inv,x\cdot u),s\inv\bigr)
    \\&=(x\inv,s\inv\cdot x\cdot u,s\inv),
  \end{align*}
  \begin{align*}
    (x\inv,s\inv\cdot x\cdot u,s\inv)\cdot (y\inv,v) &=(x\inv
    y\inv,s\inv\cdot v),
  \end{align*}
  \begin{align*}
    (x\inv y\inv,s\inv\cdot v)\cdot y &=(x\inv,y\inv\cdot s\inv\cdot
    v),
  \end{align*}
  and
  \begin{align*}
    \rho(y\inv,v)=(s(y),y\inv\cdot v,e),
  \end{align*}
  so we get
  \begin{align*}
    &\<\xi,\eta\>_R(y) \\&\quad=\int_{(\YY*\Omega)\rtimes G} \bigl\<
    \xi(x\inv y\inv,s\inv\cdot v), \\&\quad\hspace{1in}
    \eta(x\inv,y\inv\cdot s\inv\cdot v) \bigr\>_\BB
    \,d\lambda^{(s(y),y\inv\cdot v,e)}(x,u,s)
    \\&\quad\hspace{2in}\text{where $y\inv\cdot v=x\cdot u$}
    \\&\quad=\int_\YY\int_G \Bigl\< \bigl(\xi_1(x\inv y\inv,s\inv\cdot
    v),s\inv\cdot v\bigr), \\&\hspace{1in} \bigl(\eta(x\inv,y\inv\cdot
    s\inv\cdot v),y\inv\cdot s\inv\cdot v\bigr) \Bigr\>_\BB
    \,ds\,d\lambda^{s(y)}(x) \\&=\int_\YY\int_G \xi_1(x\inv
    y\inv,s\inv\cdot v)^* \eta(x\inv,y\inv\cdot s\inv\cdot v)
    \,ds\,d\lambda^{s(y)}(x) \\&=\int_\YY\int_G \xi_1(x\inv
    y\inv,s\inv\cdot yx\cdot u)^* \eta(x\inv,s\inv\cdot x\cdot u)
    \,ds\,d\lambda^{s(y)}(x).
  \end{align*}

  For $X_R$, recall the abstract formulas from \cite{proper} for the
  imprimitivity bimodule associated to an action $\alpha:G\to\aut A$
  (where $A$ is any $C^*$-algebra) that is saturated and proper with
  respect to a dense *-subalgebra $A_0$: for $f\in C_c(G,A)\subset
  A\rtimes_{\alpha,r} G$ and $\xi,\eta,\zeta\in A_0$ we have
  \begin{align*}
    f\cdot \xi&=\int_G f(s)\alpha_s(\xi)\,ds\\
    {}_{A\rtimes_{\alpha,r} G}\<\xi,\eta\>(s)&=\Delta(s)^{-1/2}\xi\alpha_s(\eta^*)\\
    \zeta\<\xi,\eta\>_{A^\alpha}&=\int_G \zeta\alpha_s(\xi^*\eta)\,ds,
  \end{align*}
  and of course the right module action of $A^\alpha$ on $A$ is given
  by right multiplication.

  In our context we have $A=C^*(\BB*\Omega)$ and
  $A_0=\Gamma_c(\BB*\Omega)$.  We take
  \[
  f\in C_c(G,\Gamma_c(\BB*\Omega))\subset
  \Gamma_c\bigl((\BB*\Omega)\rtimes G\bigr),
  \]
  and then the left module action becomes
  \begin{align*}
    &(f\cdot \xi)(y,u) \\&\quad=\int_G
    \bigl(f(s)\alpha_s(\xi)\bigr)(y,u)\,ds
    \\&\quad=\int_G\int_{\YY*\Omega}
    f(s)(x,v)\alpha_s(\xi)\bigl((x,v)\inv (y,u)\bigr)
    \,d\lambda_\YY^{r(y,u)}(x,v)\,ds \\&\quad=\int_G\int_\YY
    f(s)(x,v)\alpha_s(\xi)\bigl((x\inv,x\cdot v)(y,u)\bigr)
    \,d\lambda_\YY^{r(y)}(x)\,ds \\&\quad=\int_G\int_\YY
    f(s)(x,v)\alpha_s(\xi)(x\inv y,u) \,d\lambda_\YY^{r(y)}(x)\,ds
    \\&\quad=\int_G\int_\YY f(s)(x,v)\,s\cdot
    \Bigl(\xi\bigl(s\inv\cdot (x\inv,x\cdot v)\bigr)\Bigr)
    \,d\lambda_\YY^{r(y)}(x)\,ds \\&\quad=\int_G\int_\YY
    f(s)(x,v)\,s\cdot \bigl(\xi(x\inv,s\inv\cdot x\cdot v)\bigr)
    \,d\lambda_\YY^{r(y)}(x)\,ds \\&\quad=\int_G\int_\YY
    f(s)(x,v)\,s\cdot \bigl(\xi_1(x\inv,s\inv\cdot x\cdot
    v),s\inv\cdot x\cdot v\bigr) \,d\lambda_\YY^{r(y)}(x)\,ds
    \\&\quad=\int_G\int_\YY
    \bigl(f(s)_1(x,v),v\bigr)\bigl(\xi_1(x\inv,s\inv\cdot x\cdot
    v),x\cdot v\bigr) \,d\lambda_\YY^{r(y)}(x)\,ds
    \\&\quad=\int_G\int_\YY \bigl(f_1(x,v,s)\xi_1(x\inv,s\inv\cdot
    x\cdot v),x\cdot v\bigr) \,d\lambda_\YY^{r(y)}(x)\,ds,
  \end{align*}
  where the meaning of the notation $f_1$ is clear once we identify
  $f\in C_c(G,\Gamma_c(\BB*\Omega))$ with the corresponding section in
  $\Gamma_c((\BB*\Omega)\rtimes G)$, and thus
  \begin{align*}
    (f\cdot\xi)_1(y,u) &=\int_G\int_\YY
    f_1(x,v,s)\xi_1(x\inv,s\inv\cdot x\cdot
    v)\,d\lambda_\YY^{r(y)}(x)\,ds \\&=\int_G\int_\YY f_1(x,x\inv
    y\cdot u,s)\xi_1(x\inv,s\inv\cdot y\cdot
    u)\,d\lambda_\YY^{r(y)}(x)\,ds,
  \end{align*}
  because $y\cdot u=x\cdot v$.

  For the left inner product on $\Gamma_c(\BB*\Omega)\subset X_R$, if
  $\xi,\eta\in \Gamma_c(\BB*\Omega)$ then the inner product
  ${}_{C^*(\BB*\Omega)\rtimes_{\alpha,r} G}\<\xi,\eta\>$ lies in
  $C_c(G,\Gamma_c(\BB*\Omega))\subset \Gamma_c((\BB*\Omega)\rtimes
  G)$, and we have
  \begin{align*}
    &{}_{C^*(\BB*\Omega)\rtimes_{\alpha,r} G}\<\xi,\eta\>(y,u,s)
    \\&\quad={}_{C^*(\BB*\Omega)\rtimes_{\alpha,r}
      G}\<\xi,\eta\>(s)(y,u)
    \\&\quad=\Delta(s)^{-1/2}\bigl(\xi\alpha_s(\eta^*)\bigr)(y,u)
    \\&\quad=\Delta(s)^{-1/2}\int_{\YY*\Omega}
    \xi(x,v)\alpha_s(\eta^*)\bigl((x,v)\inv (y,u\bigr)
    \,d\lambda_{\YY*\Omega}^{r(y,u)}(x,v)
    \\&\quad=\Delta(s)^{-1/2}\int_{\YY*\Omega}
    \xi(x,v)\alpha_s(\eta^*)\bigl((x\inv,x\cdot v)(y,u\bigr)
    \,d\lambda_{\YY*\Omega}^{(r(y),y\cdot v)}(x,v)
    \\&\quad=\Delta(s)^{-1/2}\int_\YY \xi(x,v)\alpha_s(\eta^*)(x\inv
    y,u) \,d\lambda_\YY^{r(y)}(x) \\&\quad=\Delta(s)^{-1/2}\int_\YY
    \bigl(\xi_1(x,v),v\bigr)\bigl(\eta_1^*(x\inv y,s\inv\cdot
    u),u\bigr) \,d\lambda_\YY^{r(y)}(x)
    \\&\quad=\Delta(s)^{-1/2}\int_\YY
    \bigl(\xi_1(x,v),v\bigr)\bigl(\eta_1(y\inv x,s\inv\cdot x\inv
    y\cdot u)^*,u\bigr) \,d\lambda_\YY^{r(y)}(x)
  \end{align*}
  so
  \begin{align*}
    &{}_{C^*(\BB*\Omega)\rtimes_{\alpha,r} G}\<\xi,\eta\>_1(y,u,s)
    \\&\quad=\Delta(s)^{-1/2}\int_\YY \xi_1(x,x\inv y\cdot
    u)\eta_1(y\inv x,s\inv\cdot x\inv y\cdot u)^*\,d\lambda^{r(y)}(x),
  \end{align*}
  because $y\cdot u=x\cdot v$.

  For the right inner product, if $\xi,\eta,\zeta\in
  \Gamma_c(\BB*\Omega)$ then
  \begin{align*}
    &\bigl(\zeta\<\xi,\eta\>_{C^*(\BB*\Omega)^\alpha}\bigr)(y,u)
    \\&\quad=\int_G \bigl(\zeta\alpha_s(\xi^*\eta)\bigr)(y,u)\,ds
    \\&\quad=\int_G\int_{\YY*\Omega} \zeta(x,v)
    \alpha_s(\xi^*\eta)\bigl((x,v)\inv (y,u)\bigr)
    \,d\lambda^{r(y,u)}(x,v)\,ds \\&\quad=\int_G\int_{\YY*\Omega}
    \zeta(x,v) \alpha_s(\xi^*\eta)(x\inv y,u) \,d\lambda^{(r(y),y\cdot
      u)}(x,v)\,ds \\&\quad=\int_G\int_\YY \bigl(\zeta_1(x,v),v\bigr)
    \bigl((\xi^*\eta)_1(x\inv y,s\inv\cdot u),u\bigr)
    \,d\lambda^{r(y)}(x)\,ds,
  \end{align*}
  so
  \begin{align*}
    &\bigl(\zeta\<\xi,\eta\>_{C^*(\BB*\Omega)^\alpha}\bigr)_1(y,u)
    \\&\quad=\int_G\int_\YY \zeta_1(x,v) (\xi^*\eta)_1(x\inv
    y,s\inv\cdot u) \,d\lambda^{r(y)}(x)\,ds \\&\quad=\int_G\int_\YY
    \zeta_1(x,v) \int_\YY \xi_1(z\inv,s\inv\cdot x\inv y\cdot u)^*
    \eta_1(z\inv x\inv y,s\inv\cdot u)\,d\lambda^{r(x\inv y)}(z)
    \\&\quad\hspace{2in}\,d\lambda^{r(y)}(x)\,ds
    \\&\quad=\int_G\int_\YY\int_\YY \zeta_1(x,x\inv y\cdot u)
    \xi_1(z\inv,s\inv\cdot x\inv y\cdot u)^* \eta_1(z\inv x\inv
    y,s\inv\cdot u)
    \\&\quad\hspace{2in}\,d\lambda^{s(x)}(z)\,d\lambda^{r(y)}(x)\,ds.
  \end{align*}

  We will need to observe the following: for $f\in \Gamma_c(\BB)$ and
  $\xi\in \Gamma_c(\BB*\Omega)$ we have
  \begin{align*}
    \bigl(\xi\Phi(f)\bigr)_1(y,u)
    &=\bigl(\Phi(f^*)\xi^*\bigr)^*_1(y,u)
    \\&=\bigl(\Phi(f^*)\xi^*\bigr)_1(y\inv,y\cdot u)^*
    \\&=\left(\int_\YY f^*(x)\xi^*_1(x\inv y\inv,y\cdot
      u)\,d\lambda^{r(y\inv)}(x)\right)^* \\&=\left(\int_\YY
      f(x\inv)^*\xi_1(yx,x\inv\cdot u)^*\,d\lambda^{s(y)}(x)\right)^*
    \\&=\int_\YY \xi_1(yx,x\inv\cdot u)f(x\inv)\,d\lambda^{s(y)}(x).
  \end{align*}

  We now proceed to show that the identity map on
  $\Gamma_c(\BB*\Omega)$ extends to an imprimitivity-bimodule map
  $\Upsilon:X\to X_R$ (which will then be a surjection because
  $\Gamma_c(\BB*\Omega)$ is dense in $X_R$).  It suffices to show
  that, on generators in $\Gamma_c(\BB*\Omega)$, the maps $\Lambda$
  and $\Phi$ transport the inner products of the imprimitivity
  bimodule $X$ to those of $X_R$. That is, we must show that for
  $\xi,\eta\in \Gamma_c(\BB*\Omega)$ we have
  \begin{align}
    \label{left}
    \Lambda\bigl({}_{C^*(\BB*\Omega)\rtimes_\alpha
      G}\<\xi,\eta\>\bigr) &={}_{C^*(\BB*\Omega)\rtimes_{\alpha,r}
      G}\<\xi,\eta\>;
    \\
    \label{right}
    \Phi\bigl(\<\xi,\eta\>_{C^*(\BB)}\bigr)
    &=\<\xi,\eta\>_{C^*(\BB*\Omega)^\alpha}.
  \end{align}
  For \eqref{left}, first of all we have
  \begin{align*}
    &{}_{C^*((\BB*\Omega)\rtimes G)}\<\xi,\eta\>_1(y,u,s)
    \\&\quad=\int_\YY \xi_1(yz,z\inv\cdot u)\eta_1(z,s\inv\cdot
    z\inv\cdot u)^*\,d\lambda^{s(y)}(z) \\&\quad=\int_\YY
    \xi_1(z,x\inv y\cdot u)\eta_1(y\inv x,s\inv\cdot x\inv y\inv\cdot
    u)^*\,d\lambda^{r(y)}(x) \\&\quad\hspace{1in}\text{after
      substituting $x=yz$}.
  \end{align*}
  Recall from \cite[Theorem~7.1]{kmqw1} that we have an isomorphism
  \[
  C^*(\BB*\Omega)\rtimes_\alpha G\cong C^*((\BB*\Omega)\rtimes G),
  \]
  and we will actually blur the distinction between these two
  $C^*$-algebras, so that for $f\in \Gamma_c(\BB*\Omega)$ and $g\in
  C_c(G)$ the generator $i_{C^*(\BB*\Omega)}(f)i_G(g)$ of the crossed
  product $C^*(\BB*\Omega)\rtimes_\alpha G$ is identified with the
  element of $\Gamma_c((\BB*\Omega)\rtimes G)$ given by
  \[
  \bigl(i_{C^*(\BB*\Omega)}(f)i_G(g)\bigr)(x,u,t)
  =\bigl(f_1(x,u)g(t)\Delta(t)^{1/2},u,t\bigr).
  \]
  Thus ${}_{C^*((\BB*\Omega)\rtimes_\alpha G}\<\xi,\eta\>$ is the
  element of
  \begin{align*}
    C_c(G,\Gamma_c(\BB*\Omega)) &\subset C_c(G,C^*(\BB*\Omega))
    \\&\subset C^*(\BB*\Omega)\rtimes_\alpha G
  \end{align*}
  satisfying
  \begin{align*}
    &{}_{C^*((\BB*\Omega)\rtimes G)}\<\xi,\eta\>(s)_1(y,u)
    \\&\quad=\Delta(s)^{-1/2}\int_\YY \xi_1(x,x\inv y\cdot
    u)\eta_1(y\inv x,s\inv\cdot x\inv y\inv\cdot
    u)^*\,d\lambda^{r(y)}(x).
  \end{align*}
  On the other hand, we have
  \begin{align*} {}_{C^*(\BB*\Omega)\rtimes_{\alpha,r}
      G}\<\xi,\eta\>(s)
    &=\Delta(s)^{-1/2}\bigl(\xi\alpha_s(\eta)^*\bigr),
  \end{align*}
  so ${}_{C^*(\BB*\Omega)\rtimes_{\alpha,r} G}\<\xi,\eta\>$ is the
  element of
  \begin{align*}
    C_c(G,\Gamma_c(\BB*\Omega)) &\subset C_c(G,C^*(\BB*\Omega))
    \\&\subset C^*(\BB*\Omega)\rtimes_{\alpha,r} G
  \end{align*}
  satisfying
  \begin{align*}
    &{}_{C^*(\BB*\Omega)\rtimes_{\alpha,r} G}\<\xi,\eta\>(s)_1(y,u)
    \\&\quad=\Delta(s)^{-1/2} \int_\YY \xi_1(x,x\inv y\cdot
    u)\alpha_s(\eta)^*_1(x\inv y,u) \,d\lambda^{r(y)}(x)
    \\&\quad=\Delta(s)^{-1/2} \int_\YY \xi_1(x,x\inv y\cdot
    u)\alpha_s(\eta)_1(y\inv x,x\inv y\cdot u)^* \,d\lambda^{r(y)}(x)
    \\&\quad=\Delta(s)^{-1/2} \int_\YY \xi_1(x,x\inv y\cdot
    u)\eta_1(y\inv x,s\inv\cdot x\inv y\cdot u)^*
    \,d\lambda^{r(y)}(x).
  \end{align*}
  Therefore, since $\Lambda$ is the bounded extension of the identity
  map on $C_c(G,C^*(\BB*\Omega))$, we have verified \eqref{left}.

  For \eqref{right}, we will actually find it convenient to show that
  if $\zeta\in \Gamma_c(\BB*\Omega)$, then
  \[
  \zeta \Phi\bigl(\<\xi,\eta\>_{C^*(\BB)}\bigr)
  =\zeta\<\xi,\eta\>_{C^*(\BB*\Omega)^\alpha)}.
  \]
  The left side is the element of $\Gamma_c(\BB*\Omega)$ satisfying
  \begin{align*}
    &\Bigl(\zeta \Phi\bigl(\<\xi,\eta\>_{C^*(\BB)}\bigr)\Bigr)_1(y,u)
    \\&\quad=\int_\YY \zeta_1(yx,x\inv\cdot u)
    \<\xi,\eta\>_{C^*(\BB)}(x\inv) \,d\lambda^{s(y)}(x)
    \\&\quad=\int_\YY \zeta_1(yx,x\inv\cdot u) \int_\YY\int_G
    \xi_1(z\inv x,s\inv\cdot x\inv z\cdot w)^* \\&\quad\hspace{1in}
    \eta_1(z\inv,s\inv\cdot z\cdot w) \,ds\,d\lambda^{s(x\inv)}(z)
    \,d\lambda^{s(y)}(x) \\&\quad\hspace{2in}\text{where
      $\rho(w)=s(z)$; can take $w=z\inv\cdot u$}
    \\&\quad=\int_G\int_\YY\int_\YY \zeta_1(yx,x\inv\cdot u)
    \xi_1(z\inv x,s\inv\cdot x\inv\cdot u)^* \\&\quad\hspace{1in}
    \eta_1(z\inv,s\inv\cdot u) \,d\lambda^{r(x)}(z)
    \,d\lambda^{s(y)}(x)\,ds \\&\quad=\int_G\int_\YY\int_\YY
    \zeta_1(yx,x\inv\cdot u) \xi_1(z\inv x,s\inv\cdot x\inv\cdot u)^*
    \\&\quad\hspace{1in} \eta_1(z\inv,s\inv\cdot u)
    \,d\lambda^{s(y)}(z) \,d\lambda^{s(y)}(x)\,ds
    \\&\quad=\int_G\int_\YY\int_\YY \zeta_1(x,x\inv y\cdot u)
    \xi_1(z\inv y\inv x,s\inv\cdot x\inv y\cdot u)^*
    \\&\quad\hspace{1in} \eta_1(z\inv,s\inv\cdot u)
    \,d\lambda^{s(y)}(z) \,d\lambda^{r(y)}(x)\,ds
    \\&\quad=\int_G\int_\YY\int_\YY \zeta_1(x,x\inv y\cdot u)
    \xi_1(z\inv x,s\inv\cdot x\inv y\cdot u)^* \\&\quad\hspace{1in}
    \eta_1(z\inv y,s\inv\cdot u) \,d\lambda^{r(y)}(z)
    \,d\lambda^{r(y)}(x)\,ds.
  \end{align*}

  On the other hand, $\zeta\<\xi,\eta\>_{C^*(\BB*\Omega)^\alpha}$ is
  the element of $\Gamma_c(\BB*\Omega)$ satisfying
  \begin{align*}
    &\bigl(\zeta\<\xi,\eta\>_{C^*(\BB*\Omega)^\alpha}\bigr)_1(y,u)
    \\&\quad=\int_G \bigl(\zeta \alpha_s(\xi^*\eta)\bigr)_1(y,u)\,ds
    \\&\quad=\int_G\int_\YY \zeta_1(x,x\inv y\cdot u)
    \bigl(\alpha_s(\xi^*\eta)\bigr)_1(x\inv y,u)
    \,d\lambda^{r(y)}(x)\,ds \\&\quad=\int_G\int_\YY \zeta_1(x,x\inv
    y\cdot u) (\xi^*\eta)_1(x\inv y,s\inv\cdot u)
    \,d\lambda^{r(y)}(x)\,ds \\&\quad=\int_G\int_\YY \zeta_1(x,x\inv
    y\cdot u) \int_\YY \xi^*_1(z,z\inv x\inv y\cdot s\inv\cdot u)
    \\&\quad\hspace{1in} \eta_1(z\inv x\inv y,s\inv\cdot u)
    \,d\lambda^{r(x\inv y)}(z) \,d\lambda^{r(y)}(x)\,ds
    \\&\quad=\int_G\int_\YY\int_\YY \zeta_1(x,x\inv y\cdot u)
    \xi_1(z\inv,s\inv\cdot x\inv y\cdot u)^* \\&\quad\hspace{1in}
    \eta_1(z\inv x\inv y,s\inv\cdot u)
    \,d\lambda^{s(x)}(z)\,d\lambda^{r(y)}(x)\,ds
    \\&\quad=\int_G\int_\YY\int_\YY \zeta_1(x,x\inv y\cdot u)
    \xi_1(z\inv x,s\inv\cdot x\inv y\cdot u)^* \\&\quad\hspace{1in}
    \eta_1(z\inv y,s\inv\cdot u)
    \,d\lambda^{r(x)}(z)\,d\lambda^{r(y)}(x)\,ds
    \\&\quad=\int_G\int_\YY\int_\YY \zeta_1(x,x\inv y\cdot u)
    \xi_1(z\inv x,s\inv\cdot x\inv y\cdot u)^* \\&\quad\hspace{1in}
    \eta_1(z\inv y,s\inv\cdot u)
    \,d\lambda^{r(y)}(z)\,d\lambda^{r(y)}(x)\,ds.
  \end{align*}
  Therefore \eqref{right} holds, and we are done.
\end{proof}


Using a recent result of Sims and Williams, we can show that in
\thmref{Rieffel transformation} the surjection of $C^*(\BB)$ onto the
generalized fixed-point algebra can be identified with the regular
representation:

\begin{cor}\label{SW}
  Let $p:\BB\to\YY$ be a Fell bundle over a locally compact groupoid,
  and let $G$ be a locally compact group. Suppose that both $\YY$ and
  $G$ act on \(the left of\) a locally compact Hausdorff space
  $\Omega$, 
  that the action of $G$ is free and proper and commutes
  with the $\YY$-action, 
  and that the fibring map $\Omega\to \YY^0$ associated to the 
  $\YY$-action induces an identification of 
  $G\under \Omega$ with $\YY^0$,
  so that $G$ also acts freely and properly by
  automorphisms on the transformation Fell bundle $\BB*\Omega\to
  \YY*\Omega$, and we also have an associated action $\alpha:G\to\aut
  C^*(\BB*\Omega)$.  Then there is a unique isomorphism $\Xi$ making
  the diagram
  \begin{equation}\label{Xi}
    \xymatrix{
      &C^*(\BB) \ar[dl]_\Phi \ar[dr]^\Lambda
      \\
      C^*(\BB*\Omega)^\alpha \ar[rr]_-{\Xi}^-\cong
      &&C^*_r(\BB)
    }
  \end{equation}
  commute.
\end{cor}

\begin{proof}
  By \cite[Theorem~14]{SimWilReducedFell} the kernels of the regular
  representations of $C^*((\BB*\Omega)\times G)$ and $C^*(\BB)$
  correspond via the imprimitivity bimodule $X$.  By
  \cite[Theorem~7.1]{kmqw1} we have
  \[
  C^*((\BB*\Omega)\times G)\cong C^*(\BB*\Omega)\rtimes_\alpha G,
  \]
  by \cite[Example~11]{SimWilReducedFell} we have
  \[
  C^*_r((\BB*\Omega)\times G)\cong C^*(\BB*\Omega)\rtimes_{\alpha,r}
  G,
  \]
  and the regular representations
  \begin{align*}
    \Lambda:C^*((\BB*\Omega)\times G)&\to C^*_r((\BB*\Omega)\times G)
    \\
    \Lambda:C^*(\BB*\Omega)\rtimes_\alpha G&\to
    C^*(\BB*\Omega)\rtimes_{\alpha,r} G
  \end{align*}
  correspond under these isomorphisms.  Thus the kernels of the
  regular representations of $C^*(\BB*\Omega)\rtimes_\alpha G$ and
  $C^*(\BB)$ correspond via $X$. But by \thmref{Rieffel
    transformation} the kernels of the regular representation of
  $C^*(\BB*\Omega)\rtimes_\alpha G$ and of $\Phi:C^*(\BB)\to
  C^*(\BB*\Omega)^\alpha$ also correspond via $X$, so the result
  follows.
\end{proof}

It is convenient to have the following alternative version of
\corref{SW}:

\begin{cor}\label{SWA}
  Let $p:\AA\to\XX$ be a Fell bundle over a locally compact groupoid,
  and let $G$ be a locally compact group.  Suppose that $G$ acts
  freely and properly on \(the left of\) $\AA$ by automorphisms, so
  that we also have an associated action $\alpha:G\to \aut C^*(\AA)$.
  Then there is a unique isomorphism $\Xi$ making the diagram
  \begin{equation}\label{Xi A}
    \xymatrix{
      &C^*(G\under\AA) \ar[dl]_\Phi \ar[dr]^\Lambda
      \\
      C^*(\AA)^\alpha \ar[rr]_-{\Xi}^-\cong
      &&C^*_r(G\under\AA)
    }
  \end{equation}
  commute.

\end{cor}

\section{Applications}\label{applications}

\subsection{Coaction-crossed products}

Let $\BB\to G$ be a Fell bundle over a locally compact group.  Then by
\cite[Theorem~5.1]{kmqw1} we have an equivariant isomorphism
\[
\bigl(C^*(\BB\times G),\alpha\bigr)\cong \bigl(C^*(\BB)\rtimes_\delta
G,\what\delta\bigr),
\]
so the Rieffel Surjection \thmref{Rieffel transformation} in this
context can be expressed in the form
\[
(\Lambda,\Upsilon,\Phi): \bigl(C^*(\BB)\rtimes_\delta
G\rtimes_{\what\delta} G,X,C^*(\BB)\bigr) \to
\bigl(C^*(\BB)\rtimes_\delta G\rtimes_{\what\delta,r}
G,X_R,C^*_r(\BB)\bigr)
\]
Moreover, in this case we can identify the isomorphism
\[
\xymatrix@C+30pt{ \Xi:(C^*(\BB)\rtimes_\delta G)^{\what\delta}
  \ar[r]^-\cong &C^*_r(\BB) }
\]
of \eqref{Xi}: the Banach bundle $\BB\to G$ gives an equivalence
between the Fell bundles $\BB\times G\to G\times G$ and $B(e)\to
\{e\}$, and hence by the YMW Theorem we have a $C^*(\BB)\rtimes_\delta
G-B(e)$ imprimitivity bimodule $L^2(\BB)$, and hence an isomorphism
\[
\xymatrix{ C^*(\BB)\rtimes_\delta G \ar[r]^\varphi_\cong
  &\KK(L^2(\BB)).  }
\]

\begin{thm}\label{regular}
  With the above notation, the isomorphism $\Xi$ of \eqref{Xi} is the
  restriction to $(C^*(\BB)\rtimes_\delta G)^{\what\delta}$ of the
  canonical extension
  \[
  \bar\varphi:M\bigl(C^*(\BB)\rtimes_\delta G\bigr)\to \LL(L^2(\BB)).
  \]
\end{thm}

\begin{proof}
  Let
  \[
  \xymatrix{ C^*(\BB)\rtimes_\delta G \ar[r]^\cong_\theta
    &C^*(\BB\times G) }
  \]
  be the isomorphism of \cite[Theorem~5.1]{kmqw1}, and let
  \[
  \psi=\varphi\circ\theta\inv:C^*(\BB\times G)\to \LL(L^2(\BB)).
  \]
  Since $\Phi(C^*(\BB))=C^*(\BB\times G)^\alpha$ and
  $\Lambda(C^*(\BB))=C^*_r(\BB)$, it suffices to show that the diagram
  \begin{equation}\label{Phi=Lambda2}
    \xymatrix{
      &C^*(\BB)
      \ar[dl]_\Phi
      \ar[dr]^\Lambda
      \\
      M(C^*(\BB\times G))
      \ar[rr]_-{\bar\psi}
      &&\LL(L^2(\BB))
    }
  \end{equation}
  commutes.  Let's recall that for $g\in \Gamma_c(\BB\times G)$ and
  $\xi\in \Gamma_c(\BB)$ we have
  \begin{align*}
    \bigl(\psi(g)\xi\bigr)(x) &=\int_{G\rtimes_{\lt} G}
    g_1(y,u)\xi\bigl((y,u)\inv\cdot x\bigr) \,d\lambda_{G\rtimes_{\lt}
      G}^{r(x)}(y,u) \\&=\int_G g_1(y,y\inv x)\xi(y\inv x)\,dy.
  \end{align*}

  It suffices to check commutativity of the diagram on functions $f\in
  \Gamma_c(\BB)$, and it suffices to check the values of
  $\bar\psi\circ\Phi(f)$ and $\Lambda(f)$ on vectors in $\ell^2(\BB)$
  of the form $\psi(g)\xi$ for $g\in \Gamma_c(\BB\times G)$ and
  $\xi\in \ell^2(\BB)$:
  \begin{align*}
    \bigl(\bar\psi(\Phi(f))\psi(g)\xi\bigr)(x)
    &=\Bigl(\psi\bigl(\Phi(f)g\bigr)\xi\Bigr)(x) \\&=\int_G
    \bigl(\Phi(f)g\bigr)_1(y,y\inv x)\xi(y\inv x)\,dy \\&=\int_G\int_G
    f(s)g_1(s\inv y,y\inv x)\,ds\xi(y\inv x)\,dy \\&=\int_G\int_G
    f(s)g_1(s\inv y,y\inv x)\xi(y\inv x)\,dy\,ds \\&=\int_G f(s)\int_G
    g_1(y,y\inv s\inv x)\xi(y\inv s\inv x)\,dy\,ds \\&=\int_G
    f(s)\bigl(\psi(g)\xi\bigr)(s\inv x)\,ds
    \\&=\bigl(\Lambda(f)\psi(g)\xi\bigr)(x).
  \end{align*}
  Thus \eqref{Phi=Lambda2} commutes.
\end{proof}

We can deduce from the above that, as one would expect, the regular
representation of $C^*(\BB)$ is a normalization:

\begin{cor}\label{normalization}
  Let $\BB\to G$ be a Fell bundle over a locally compact group, and
  let $\delta$ be the canonical coaction of $G$ on $C^*(\BB)$.  Then
  there is a unique coaction $\delta^n$ of $G$ on $C^*_r(\BB)$ such
  that the regular representation
  \[
  \xymatrix{ (C^*(\BB),\delta) \ar[r]^-\Lambda &(C^*_r(\BB),\delta^n).
  }
  \]
  is a normalization of $\delta$.
\end{cor}

\begin{proof}
  We will show that the diagram
  \begin{equation}\label{j_B=Phi}
    \xymatrix{
      &C^*(\BB)
      \ar[dl]_{j_\BB}
      \ar[dr]^\Phi
      \\
      M\bigl(C^*(\BB)\rtimes_\delta G\bigr)
      \ar[rr]_-{\bar\theta}^-\cong
      &&
      M\bigl(C^*(\BB\times G)\bigr)
    }
  \end{equation}
  commutes.  It will then follow from \thmref{regular} that the
  diagram
  \[
  \xymatrix{ &C^*(\BB) \ar[dl]_{j_\BB} \ar[dr]^\Lambda
    \\
    j_\BB(C^*(\BB)) \ar[rr]_-{\Xi\circ\theta}^-\cong && C^*_r(\BB) }
  \]
  commutes.  Since there is a unique coaction $\ad j_G$ on
  $j_\BB(C^*(\BB))$ such that
  \[
  j_\BB:\bigl(C^*(\BB),\delta\bigr)\to \bigl(j_\BB(C^*(\BB)),\ad
  j_G\bigr)
  \]
  is a normalization, this will complete the proof.  The following
  computation implies that \eqref{j_B=Phi} commutes: for $f,g\in
  \Gamma_c(\BB)$, $h\in C_c(G)$, and $k\in \Gamma_c(\BB\times G)$ we
  have
  \begin{align*}
    &\biggl(\bar\theta\circ
    j_\BB(f)\Bigl(\theta\bigl(j_\BB(g)j_G(h)\bigr)*k\Bigr)\biggr)(s,t)
    \\&\quad=\Bigl(\theta\bigl(j_\BB(f)j_\BB(g)j_G(h)\bigr)*k\Bigr)(s,t)
    \\&\quad=\Bigl(\theta\bigl(j_\BB(f*g)j_G(h)\bigr)*k\Bigr)(s,t)
    \\&\quad=\Bigl(\bigl(\Delta^{1/2}(f*g)\boxtimes
    h\bigr)*k\Bigr)(s,t) \\&\quad=\int_G
    \bigl(\Delta^{1/2}(f*g)\boxtimes h\bigr)(r,r\inv st)k(r\inv
    s,t)\,dr \\&\quad=\int_G \Delta(r)^{1/2}(f*g)(r)h(r\inv st)k(r\inv
    s,t)\,dr \\&\quad=\int_G \Delta(r)^{1/2}\int_G f(u)g(u\inv
    r)\,du\,h(r\inv st)k(r\inv s,t)\,dr \\&\quad=\int_G f(u)\int_G
    \Delta(r)^{1/2}g(u\inv r)h(r\inv st)k(r\inv s,t)\,dr\,du
    \\&\quad=\int_G f(u)\int_G \Delta(r)^{1/2}g(r)h(r\inv u\inv
    st)k(r\inv u\inv s,t)\,dr\,du \\&\quad=\int_G
    f(u)\bigl((\Delta^{1/2}g\boxtimes h)*k\bigr)(u\inv s,t)\,du
    \\&\quad=\int_G
    f(u)\Bigl(\theta\bigl(j_\BB(g)j_G(h)\bigr)*k\Bigr)(u\inv s,t)\,du
    \\&\quad=\Bigl(\Phi(f)\Bigl(\theta\bigl(j_\BB(g)j_G(h)\bigr)*k\Bigr)\biggr)(s,t).
    \qedhere
  \end{align*}
\end{proof}

\begin{rem}
  We can interpret the above as confirmation that Katayama duality for
  normal coactions is a quotient of Katayama duality for maximal ones:
  $X$ can be viewed as a $C^*(\BB)\rtimes_\delta
  G\rtimes_{\what\delta} G-C^*(\BB)$ imprimitivity module, and $X_R$
  as a $C^*_r(\BB)\rtimes_{\delta^n} G\rtimes_{\what{\delta^n},r}
  G-C^*_r(\BB)$ imprimitivity module, and then the Rieffel Surjection
  $\Upsilon:X\to X_R$ of \thmref{Rieffel transformation} is compatible
  with the regular representations $C^*(\BB)\rtimes_\delta
  G\rtimes_{\what\delta} G\to C^*_r(\BB)\rtimes_{\delta^n}
  G\rtimes_{\what{\delta^n},r} G$ and $C^*(\BB)\to C^*_r(\BB)$.  This
  follows from \thmref{normalization}: we only need to observe the
  following equivariant isomorphisms:
  \begin{equation}\label{equivariant}
    \bigl(C^*(\BB\times G),\alpha\bigr)
    \cong \bigl(C^*(\BB)\rtimes_\delta G,\what\delta\bigr)
    \cong \bigl(C^*_r(\BB)\rtimes_{\delta^n} G,\what{\delta^n}\bigr),
  \end{equation}
  which pass to the crossed products, and hence to the reduced crossed
  products, and then apply \thmref{Rieffel transformation}.

  The isomorphisms \eqref{equivariant} imply the known result (see,
  e.g., \cite[Theorem~4.1]{cldx})
  \begin{equation}\label{fixed}
    j_\BB(C^*(\BB))=\bigl(C^*(\BB)\rtimes_\delta G\bigr)^{\what\delta},
  \end{equation}
  and hence we have a commutative diagram
  \begin{equation}\label{j_B=Lambda}
    \xymatrix{
      &C^*(\BB)
      \ar[dl]_{j_\BB}
      \ar[dr]^\Lambda
      \\
      \bigl(C^*(\BB)\rtimes_\delta G\bigr)^{\what\delta}
      \ar[rr]_-{\Xi\circ\theta}^-\cong
      &&
      C^*_r(\BB)
    }
  \end{equation}
\end{rem}

\subsection{Actions}

The following corollary appears to be new in its full generality; it
is certainly well-known in the special case that $G$ is compact.
Also, the case $A=\C$ (and arbitrary $G$) is
\cite[Example~2.1]{proper}.  Echterhoff and Emerson prove a special
case \cite[Theorem~2.14]{EchterhoffEmerson} where $A$ is fibered over
a proper $G$-space.  Our techniques do not require any hypotheses on
the action of $G$ on $A$.

\begin{cor}\label{fixed=regular}
  If $\beta:G\to \aut A$ is an action on a $C^*$-algebra, then the
  tensor-product action $\beta\otimes\ad\rho$ of $G$ on $A\otimes
  \KK(L^2(G))$ is saturated and proper in Rieffel's sense, and the
  generalized fixed point algebra $(A\otimes
  \KK(L^2(G)))^{\beta\otimes\ad\rho}$ is isomorphic to the reduced
  crossed product $A\rtimes_{\beta,r} G$.
\end{cor}

\begin{proof}
  In diagram~\eqref{j_B=Lambda}, we take the Fell bundle $\BB\to G$ to
  be the semidirect-product bundle
  \[
  A\rtimes G\to G.
  \]
  Then we have an equivariant isomorphism
  \[
  \bigl(C^*(\BB),\delta\bigr)\cong \bigl(A\rtimes_\beta
  G,\what\beta\bigr).
  \]

  Thus we have a commutative diagram
  \[
  \xymatrix{ &A\rtimes_\beta G \ar[dl]_{j_{A\rtimes_\beta G}}
    \ar[dr]^\Lambda
    \\
    \bigl(A\rtimes_\beta G\rtimes_{\what\beta}
    G\bigr)^{\what{\what\beta}} \ar[rr]_-{\Xi\circ\theta}^-\cong &&
    A\rtimes_{\beta,r} G.  }
  \]
  The result now follows from the equivariant isomorphism of
  Imai-Takai duality:
  \[
  \bigl(A\rtimes_\beta G\rtimes_{\what\beta}
  G,\what{\what\beta}\,\bigr) \cong \bigl(A\otimes
  \KK(L^2(G)),\beta\otimes \ad\rho\bigr).
  \]
  Note: there is a subtlety here: we have freely passed from
  equivariant isomorphism between proper and saturated actions to an
  isomorphism between the generalized fixed-point algebras; but
  Rieffel's generalized fixed-point algebras depend upon the choice of
  a suitable dense *-subalgebra. However, there is no problem in our
  case, because we always choose the ``canonical'' subalgebra
  associated to the obvious nondegenerate equivariant homomorphism of
  $C_0(G)$ into the multiplier algebra; then the isomorphsims follow
  from \cite[Proposition~2.6]{cldx}, modulo the correction in
  \cite[Proposition~2.4]{aHKRWFixed}.
\end{proof}

\subsection{$C^*$-bundles}

Here we specialize to the case where $\XX=X$ is a \emph{space} and
$\AA$ is just a $C^*$-bundle over $X$, so that
$C^*(\AA)=\Gamma_0(\AA)$.  Then the orbit bundle is the $C^*$-bundle
$\BB\to Y$, where $Y=G\under X$.

\begin{prop}\label{C*bundle}
  If a group $G$ acts freely and properly on a $C^*$-bundle $\AA\to X$
  over a space $X$, then the surjection
  \[
  \Phi:C^*(G\under\AA)\to C^*(\AA)^\alpha
  \]
  from \thmref{Rieffel} is an isomorphism.
\end{prop}

\begin{proof}
  With the notation used in \thmref{Rieffel transformation}, the
  groupoid $\YY=G\under\XX$ coincides with its unit space
  $Y=\YY^0=\YY$, which of course acts trivially on the space
  $X=\XX^0=\XX$, consequently the transformation groupoid $\YY*\XX^0$
  can be identified with $X$. The transformation bundle
  $\AA=\BB*\XX^0$ can be identified with $\BB*X$, and every section
  $a\in \Gamma_c(\BB*X)$ is of the form
  \[
  a(x)=\bigl(a_1(q(x)),x\bigr),
  \]
  where $a_1\in \Gamma_c(\BB)$.  For $f\in \Gamma_c(\BB)$ and $a\in
  \Gamma_c(\BB*X)$ we have
  \begin{align*}
    \bigl(\Phi(f)a\bigr)_1(x) &=f(q(x))a_1(x),
  \end{align*}
  so
  \begin{align*}
    \bigl(\Phi(f)a\bigr)(x) &=\bigl(f(q(x))a_1(x),x\bigr)
    \\&=\bigl(f(q(x)),x\bigr)\bigl(a_1(x),x\bigr) \\&=q^*(f)(x)a(x),
  \end{align*}
  where we define $q^*(f)\in \Gamma_b(\BB*X)$
  \[
  \bigl(\Phi(f)a\bigr)(x)=\bigl(f(q(x)),x\bigr).
  \]
  Thus $\Phi(f)$ acts on $\Gamma_c(\BB*X)$ by pointwise multiplication
  by the continuous bounded section $\Phi(f)\in \Gamma_b(\BB*X)$ given
  by
  \[
  \Phi(f)(x)=\bigl(f(q(x)),x\bigr).
  \]
  Since $\Gamma_b(\BB*X)$ embeds isometrically into the multiplier
  algebra $M(\Gamma_0(\BB*X))$, it follows that $\Phi:\Gamma_c(\BB)\to
  M(\Gamma_0(\BB*X))$ is isometric, and hence the extension to
  $\Gamma_0(\BB)=C^*(\BB)$ is an isomorphism onto its image
  $C^*(\BB*X)^\alpha$.
\end{proof}

\propref{C*bundle} and \thmref{Rieffel} immediately imply the
following corollary, which is surely folklore, although we could not
find it in the literature:

\begin{cor}\label{C*bundle cor}
  If a group $G$ acts freely and properly on a $C^*$-bundle $\AA\to X$
  over a space $X$, then the regular representation
  \[
  \Lambda:\Gamma_0(\AA)\rtimes_\alpha G\to
  \Gamma_0(\AA)\rtimes_{\alpha,r} G
  \]
  is an isomorphism.
\end{cor}

\thmref{C*bundle} and \corref{C*bundle cor} allow us to recover
\cite[Theorem~2.2]{RWpullback}:

\begin{cor}
  If a group $G$ acts freely and properly on a locally compact
  Hausdorff space $X$ and also on a $C^*$-algebra $A$, then the
  crossed product $C_0(X,A)\rtimes G$ is Morita equivalent to the
  generalized fixed point algebra $C_0(X,A)^\alpha$.
\end{cor}

\begin{proof}
  This follows by applying the above results to the trivial
  $C^*$-bundle $A\times X\to X$, since $\Gamma_0(A\times X)\cong
  C_0(X,A)$.
\end{proof}

\providecommand{\bysame}{\leavevmode\hbox to3em{\hrulefill}\thinspace}
\providecommand{\MR}{\relax\ifhmode\unskip\space\fi MR }
\providecommand{\MRhref}[2]{%
  \href{http://www.ams.org/mathscinet-getitem?mr=#1}{#2} }
\providecommand{\href}[2]{#2}

\end{document}